%% file: motivic-toric-arxiv-v2.tex
\documentclass[a4paper,11pt]{amsart}
\usepackage[utf8]{inputenc}
\usepackage[stretch=10]{microtype}
\usepackage{mathpazo}

\usepackage{hyphenat} %
\usepackage[all,cmtip,2cell]{xy}
\UseTwocells
\xyoption{2cell}
\usepackage{amsfonts}
\usepackage{amsthm}
\usepackage{amsmath}
\usepackage{amssymb}
\usepackage[mathscr]{euscript}
\usepackage{enumerate}
\usepackage{mathtools}
\usepackage{url}
\usepackage{tikz-cd}
\usetikzlibrary{decorations.pathmorphing,shapes}
\usepackage{calc}
\usepackage{geometry}
\usepackage[pdftex]{hyperref} %

\usepackage[nameinlink,capitalise,noabbrev]{cleveref}
\theoremstyle{definition} \newtheorem{defn}[equation]{Definition}
\theoremstyle{definition} \newtheorem{prop-def}[equation]{Proposition-Definition}
\theoremstyle{plain} \newtheorem{thm}[equation]{Theorem}
\theoremstyle{plain} \newtheorem{lemma}[equation]{Lemma}
\theoremstyle{plain} \newtheorem{prop}[equation]{Proposition}
\theoremstyle{plain} \newtheorem{cor}[equation]{Corollary}
\theoremstyle{plain} 
\theoremstyle{plain} \newtheorem*{thm*}{Theorem}
\theoremstyle{remark} \newtheorem{rmk}[equation]{Remark}
\theoremstyle{remark} 
\theoremstyle{definition} 
\theoremstyle{remark} \newtheorem{ex}[equation]{Example}
\numberwithin{equation}{section}

\newtheoremstyle{TheoremNum}
{}{}              %
{\itshape}                      %
{}                              %
{\bfseries}                     %
{.}                             %
{ }                             %
{\thmname{#1}\thmnote{ \bfseries #3}}%
\theoremstyle{TheoremNum}
\newtheorem{thmn}{Theorem}

\newcommand{\Bl}{\mathrm{Bl}}

\newcommand{\C}{\mathbb{C}}

\newcommand{\Z}{\mathbb{Z}}

\newcommand{\R}{\mathbb{R}}

\newcommand{\bbP}{\mathbb{P}}

\newcommand{\id}{\mathrm{id}}
\newcommand{\CH}{\mathrm{CH}}
\newcommand{\cone}{\mathrm{Cone}}

\newcommand{\Proj}{\mathrm{Proj}}

\newcommand{\coker}{\mathrm{coker}}

\newcommand{\Spec}{\mathrm{Spec}}

\newcommand{\xs}{X_{\widetilde\Sigma}}
\newcommand{\Hom}{\mathrm{Hom}}

\renewcommand{\int}{\mathrm{int}}

\newcommand{\adj}[4]{\xymatrix{#1 \ar@<.2pc>[r]^-{#3} & \ar@<.2pc>[l]^-{#4} #2 }}
\newcommand{\eqv}[4]{\xymatrix{#1 \ar@<.5pc>[r]^-{#3} & \ar@<.5pc>[l]^-{#4}_-\sim #2 }}

\newcommand{\hbm}{H^{\mathrm{BM}}}

\newcommand{\Conv}{\mathrm{Conv}}

\newcommand{\Corr}{\mathrm{Corr}}

\newcommand{\cl}{\mathrm{cl}}

\renewcommand{\subset}{\subseteq}

\DeclareMathAlphabet{\mathbbe}{U}{bbold}{m}{n}

\newcommand{\be}{\begin{enumerate}}
\newcommand{\ee}{\end{enumerate}}

\definecolor{bubblegum}{rgb}{0.99, 0.76, 0.8}

\DeclareUnicodeCharacter{00A0}{ } %

\definecolor{dark-red}{rgb}{0.6,.15,.15}
\definecolor{dark-blue}{rgb}{.1,.1,.65}
\definecolor{dark-green}{rgb}{.1,.65,.1}
\hypersetup{pdfauthor={Bruno Stonek},colorlinks=true, citecolor=dark-blue, urlcolor=dark-green, linkcolor=dark-red,breaklinks=true,hypertexnames=false}
\title{The motive of a variety with cellular resolution of singularities}
\author{Bruno Stonek}
\address{Faculty of Mathematics, Informatics and Mechanics, University of Warsaw, ul. Banacha 2, 02-097 Warszawa, Poland}
\email{bstonek@mimuw.edu.pl}

\begin{document}
\date{\today}

\begin{abstract} A complex variety $X$ admits a \emph{cellular resolution of singularities} if there exists a resolution of singularities $\widetilde X\to X$ such that its exceptional locus as well as $\widetilde X$ and the singular locus of $X$ admit a cellular decomposition. We give a concrete description of the motive with compact support of $X$ in terms of its Borel--Moore homology, under some mild conditions. We give many examples, including rational projective curves and toric varieties of dimension two and three.
\end{abstract}

\maketitle

\section{Introduction}

The theory of motives dates back to Grothendieck in the 1960s. It has seen tremendous advances, notably since the work of Voevodsky, who introduced a category $DM$ (standing for ``derived category of motives'') satisfying many properties common to all homology theories, such as homotopy invariance, Mayer--Vietoris, Gysin sequences, etc. We work with Voevodsky's category of motives over the field of complex numbers: this restriction simplifies some technical aspects of the theory.

Any complex variety $X$, smooth or not, determines a motive $M(X)$: it is an object in $DM$. This extends to a functor $M$, and the philosophy is that every cohomology theory should factor through it: this immediately suggests that $M(X)$ contains much interesting and subtle information about $X$.  It is  thus natural to ask for a complete description of it.

We will also consider the compactly-supported variant $M^c(X)$: it coincides with $M(X)$ when $X$ is complete. The object $M^c(X)$ determines, for example, the Chow groups of $X$, via the formula $\CH_i(X)\cong \Hom_{DM}(\Z\{i\}, M^c(X))$ (\cite[Section 2.2]{voe00}). The Chow groups, in turn, are linked to the easier Borel--Moore homology groups via the cycle map $\CH_*(X)\to \hbm_{2*}(X)$. We can turn this around and ask: when is it enough to know the Borel--Moore homology to describe the motive of the variety?

Consider \emph{cellular varieties}, by which we mean varieties which admit a filtration by closed subvarieties such that the relative differences are disjoint unions of affine spaces. Their cycle map is an isomorphism. The odd Borel--Moore homology groups are zero, and the even ones are free, with rank given by the number of cells in each dimension. It was proven by Kahn \cite{kah99} that the motive with compact support of a cellular variety is completely determined by its Chow groups, so it is also completely determined by its Borel--Moore homology.

We generalize this by proving that the Borel--Moore homology of $X$ completely determines $M^c(X)$ in the case when $X$ admits a cellular resolution of singularities, see \cref{lem-mothom}.(\ref{item-cofiber}). By this we mean that there exists a resolution of singularities $\widetilde X\to X$ such that its exceptional locus as well as $\widetilde X$ and the singular locus of $X$ are all cellular. This description of $M^c(X)$ may not be all that explicit, but adding an extra hypothesis on the Borel--Moore homology of $X$ the result becomes very explicit:

\begin{thmn}[\ref{lem-mothom}%
] Let $X$ be a variety admitting a cellular resolution of singularities. Suppose that, for all $i=0,\dots,\dim(X)$ we have that $\hbm_{2i}(X)$ is free or $\hbm_{2i+1}(X)$ is zero. %
Then there is an isomorphism in $DM$
\[M^c(X)\cong \bigoplus_i \hbm_{2i}(X)\{i\} \oplus \hbm_{2i+1}(X)\{i\}[1].\]
\end{thmn}

We give many concrete examples, especially of motives of singular varieties. This was our initial motivation, since the literature contains few explicit computations of such motives. 

We first deal with projective rational curves: we prove that their motive is determined by the number of singularities it has as well as their number of branches (see \cref{prop-curve}).

We then turn to toric surfaces determined by fans in the plane. If the fan is complete, so that the surface is projective, the motive is determined by the number of rays as well as the index of the sublattice spanned by the minimal generators of the rays, the latter determining the torsion (\cref{thm-motive-surface}). If the fan is generated by a single cone, so that the variety is affine, then the motive with compact support is determined by the determinant of the two minimal generators of the cone (\cref{prop-mc-affine}). We can also describe the motive with compact support of quasiprojective toric surfaces, under some additional conditions on the fan (see \cref{ex-quasiproj}).

Finally, we turn to motives of toric varieties of higher dimension. Here, we may have toric varieties which do not admit a cellular resolution of singularities, because the singular locus is not cellular (cf. \cref{ex-bbfk}), but we can describe the motive (or motive with compact support) of many varieties of interest. We observe that the motive of the weighted $n$-dimensional projective space $\bbP(1,\dots,1,k)$, $k\geq 2$, is the same as the motive of $\bbP^n$.

\subsection{Notations and conventions}

By \emph{variety} we will mean a complex variety, i.e. a finite-type, separated, reduced scheme over $\Spec(\C)$, which we do not assume to be irreducible. %

We adopt the conventions of toric geometry from \cite{cox}. In particular, $N$ is a lattice (a free abelian group of finite rank) with dual lattice $M$; a \emph{cone} is a strongly convex, rational, polyhedral cone, which we further assume to be non-zero; %
we denote the normal affine toric variety of a cone in $N_\R$ by $U_\sigma$, where $N_\R\coloneqq N\otimes_\Z \R$. The normal toric variety of a fan $\Sigma$ is denoted $X_\Sigma$. The standard pairing between $M_\R$ and $N_\R$ is denoted by $\langle-,-\rangle$. We denote by $\{e_i\}_i$ the standard basis of $\R^n$.

We denote by $H_*$ singular homology with integral coefficients, and $\hbm_*$ denotes integral Borel--Moore homology. For any $i\geq 0$, we let $\CH_i(X)$ denote the $i$-th Chow group of the variety $X$, i.e. the group of cycles of dimension $i$ modulo linear equivalence.

We denote by $DM$ the triangulated category of (effective) %
Voevodsky motives over $\C$. In more detail: Let $\Corr$ denote the category whose objects are the smooth separated schemes of finite type over $\C$ and whose morphisms from $X$ to $Y$ are elements of the free abelian group generated by elementary correspondences.  Endow it with the Nisnevich topology and consider the derived category %
of the abelian category of abelian group-valued sheaves over it, i.e. Nisnevich sheaves with transfers. %
The triangulated category $DM$ is its $\mathbb{A}^1$-localization; %
by a \emph{triangle} in $DM$ we mean a distinguished/exact triangle.

Put differently, $DM$ is the homotopy category of the category of chain complexes of Nisnevich sheaves with transfers, endowed with an $\mathbb{A}^1$-local descent %
model structure;  %
we will, for example, take cofibers or homotopy pushouts in here. %
We denote by $M(X)$ the motive of a scheme $X$ over $\C$, not necessarily smooth, and we let $\Z\coloneqq M(\Spec(\C))$. We denote by $\Z(1)\in DM$ the Tate motivic complex, %
and for an abelian group $A$ and $n\geq 0$, we let $A\{n\}\coloneqq A\otimes \Z(n)[2n]\in DM$. A \emph{pure Tate motive} in $DM$ is a finite direct sum of motives of the form $\Z\{n\}$. We shall also make use of motives with compact support, denoted $M^c(X)$. 
See \cite{mvw}, \cite{cisinski-deglise}, \cite{rondigs-ostvaer} for details. %

\subsection{Acknowledgments}
The author would like to thank Joachim Jelisiejew for introducing him to toric varieties and motives, and for his assistance in this project. He would also like to thank Neeraj Deshmukh, Andrzej Weber and Magdalena Zielenkiewicz for the useful discussions.

\section{The main result}

\begin{defn} \cite[1.9.1]{fulton-intersection} A variety \emph{admits a cellular decomposition} (or simply: \emph{is cellular}) if there exists a filtration $X=X_n\supset X_{n-1}\supset \cdots \supset X_0 \supset X_{-1}=\emptyset$ by closed subvarieties with each $X_{i+1}\setminus X_i$ a disjoint union of varieties isomorphic to affine spaces $\mathbb{A}^{n_{ij}}$. %
\end{defn}

\begin{rmk} \label{rmk-bmcell} The Borel--Moore homology of cellular varieties is zero in odd degrees and free abelian in even degrees $2k$, of rank the number of cells of dimension $k$, that is, the sum of the $n_{i_j}$ which are equal to $k$.
See \cite[19.1.11]{fulton-intersection}, \cite[2.3]{rossello}, \cite[3.6.3]{petersen-tosteson}.
\end{rmk}

\begin{ex} \label{ex-cellular}
\begin{enumerate}
\item A zero-dimensional variety is cellular.
\item A smooth projective toric variety is cellular thanks to the Białynicki-Birula decomposition \cite{bb1}, \cite{bb2}, %
since we can choose a $\C^*$-action with a finite set of fixed points. %
\item \label{item-cellular-3} %
In fact, one can get away with less: Let $X$ be a smooth quasiprojective variety with an action of $\C^*$ such that the set of fixed points is finite. Then it suffices that the limits $\lim_{t\to 0}t\cdot x$ exist for all $x\in X$ to ensure  the existence of such a Białynicki-Birula decomposition. %
This follows from Sumihiro's theorem, for example. %
So $X$ is cellular under these hypotheses. 
\end{enumerate}
\end{ex}

We shall mostly be concerned with the following varieties:

\begin{defn} Let $X$ be a variety. We say $X$ \emph{admits a cellular resolution of singularities} if the singular locus of $X$ is cellular and there exists a resolution of singularities $\widetilde X\to X$ such that both its exceptional locus and $\widetilde X$ are cellular.
\end{defn}

\begin{prop} \label{lem-mv}Let $X$ be a %
variety, %
$Z$ be its singular locus and $\widetilde X\to X$ be a resolution of singularities with exceptional locus $E$, so that we have a pullback diagram
\[\begin{tikzcd}
	E & {\widetilde X} \\
	Z & X.
	\arrow["j", hook, from=1-1, to=1-2]
	\arrow["p"', from=1-1, to=2-1]
	\arrow["f", from=1-2, to=2-2]
	\arrow["i"', hook, from=2-1, to=2-2]
\end{tikzcd}\]
Then the maps in the diagram induce a long exact sequence in Borel--Moore homology groups
\[\cdots \to \hbm_{n+1}(X)\to \hbm_n(E)\to \hbm_n(Z)\oplus \hbm_n(\widetilde X) \to \hbm_n(X) \to \hbm_{n-1}(E) \to \cdots\] 
and a triangle in $DM$:
\[M^c(E)\to M^c(Z)\oplus M^c(\widetilde X) \to M^c(X)\to.\]
\begin{proof}
The first statement is classical (and it can be proven similarly as below, using the chain complexes of Borel--Moore homology).  %
The version for $M$ (not $M^c$) of the second statement is \cite[14.5.3]{mvw}. %
For $M^c$, we can give a simple proof. Since $M^c$ admits localization sequences \cite[16.15]{mvw}, we apply them to the closed subvarieties $E\subset \widetilde X$ and $Z\subset X$.  Let $k:\widetilde X\setminus E \to \widetilde X$ and $r:X\setminus Z\to X$ be the inclusions, and $F:\widetilde X\setminus E \to X\setminus Z$ be the induced map by $f$, so $F$ is an isomorphism. %
Thus, we get two distinguished triangles and a morphism between them: 
\[\begin{tikzcd}
	{M^c(E)} & {M^c(\widetilde X)} & {M^c(\widetilde X\setminus E)} & {} \\
	{M^c(Z)} & {M^c(X)} & {M^c(X\setminus Z)} & {}
	\arrow["{j_*}", from=1-1, to=1-2]
	\arrow["{p_*}", from=1-1, to=2-1]
	\arrow["{k^*}", from=1-2, to=1-3]
	\arrow["{f_*}", from=1-2, to=2-2]
	\arrow[from=1-3, to=1-4]
	\arrow["{F_*}", from=1-3, to=2-3]
	\arrow["{i_*}"', from=2-1, to=2-2]
	\arrow["{r^*}"', from=2-2, to=2-3]
	\arrow[from=2-3, to=2-4]
\end{tikzcd}\]
The right square commutes by the same argument as in \cite[2.5]{nie}. %
The map $F_*$ %
is an isomorphism in $DM$. %
Therefore, the square on the left is a homotopy pushout, which equivalently gives us the desired triangle.
\end{proof}
\end{prop}

Before we come to the main theorem, let us isolate some simple homological algebra remarks that will be used in the proof. %

\begin{rmk} \label{rmk-triangles}
\begin{enumerate}
\item Let $0\to A \to B \stackrel{g}{\to} C\to D\to 0$ be an exact sequence of abelian groups . It is the homology long exact sequence of the triangle $B \stackrel{g}{\to} C\to \cone(g)\to$ in $D(\Z)$.
\item For a map of abelian groups $h:X\to Y$, denote by $[h]$ the chain complex having $X$ in degree 1, $Y$ in degree 0, and the map $h$ between them. If $B\stackrel{g}{\to} C$ is a morphism of abelian groups and $\coker(g)$ is free, then $[g]$ is quasi-isomorphic to $[0:\ker(g)\to \coker(g)]$. Indeed, since $\coker(g)$ is free there is a splitting $q:\coker(g)\to C$ which allows us to define a quasi-isomorphism between the two complexes. Notice that $[0:\ker(g)\to \coker(g)]$ is $\coker(g)\oplus \ker(g)[1]$.
\item In conclusion, if $0\to A \to B \stackrel{g}{\to} C\to D\to 0$ is an exact sequence of abelian groups and $D$ is free, then we have a triangle in $D(\Z)$:
\[B\stackrel{g}{\to} C \to D\oplus A[1]\to.\]
\end{enumerate}
\end{rmk}

\begin{thm} \label{lem-mothom}
\begin{enumerate}
\item \label{item-mot1} Let $X$ be %
a %
cellular variety. 
There is an isomorphism in $DM$
\[M^c(X)\cong \bigoplus_{i=0}^{\dim X} \hbm_{2i}(X)\{i\}\]
natural in %
cellular $X$ and proper maps between them.
\item   \label{item-dos}
Let $X$ be a variety admitting a cellular resolution of singularities with diagram
\[\begin{tikzcd}
	E & {\widetilde X} \\
	Z & X.
	\arrow[hook, from=1-1, to=1-2]
	\arrow[from=1-1, to=2-1]
	\arrow[from=1-2, to=2-2]
	\arrow[hook, from=2-1, to=2-2]
\end{tikzcd}\]
\begin{enumerate}[(a)]
\item \label{item-cofiber} The motive with compact support $M^c(X)$ is isomorphic to the cofiber of the induced map in $DM$
\[\bigoplus_i \hbm_{2i}(E)\{i\} \to \bigoplus_i \hbm_{2i}(Z)\{i\}\oplus \hbm_{2i}(\widetilde X)\{i\}.\]
\item \label{item-mot2} \label{item-mot3}
If for all $i=0,\dots, \dim(X)$ we have that $\hbm_{2i}(X)$ is free abelian or $\hbm_{2i+1}(X)=0$, then
\[M^c(X)\cong \bigoplus_i \hbm_{2i}(X)\{i\} \oplus \hbm_{2i+1}(X)\{i\}[1].\]
\end{enumerate}
\item In all the results above, if $X$ is complete, then we can replace $M^c$ by $M$ and $\hbm$ by $H$.
\end{enumerate}
\begin{proof}
\begin{enumerate}
\item
A flat map $f:X\to Y$ of relative dimension $r$ induces a map $f^*:M^c(Y)\{r\}\to M^c(X)$, essentially by flat pullback \cite[16.23]{mvw}. Taking as $f$ the map $X\to \Spec(\C)$, we obtain a map $\cl_X:\Z\{n\}\to M^c(X)$, the fundamental class of $X$ in the $(2n,n)$-Borel--Moore motivic homology group of $X$ (see \cite[2.12]{nie}). For every $i\geq 0$, we can now define a map
\begin{equation}\label{eq-bm}\CH_i(X)\to \Hom(\Z\{i\}, M^c(X)).\end{equation} %
Indeed, we can take the class of the irreducible closed subvariety $Z$ to the composition
\[\xymatrix{\Z\{i\} \ar[r]^-{\cl_Z} & \ar[r]^-{i_*} M^c(Z) & M^c(X),}\]
where $i:Z\to X$ is the inclusion: this is the motivic Borel--Moore fundamental class of $Z$ in $X$. %
Taking adjoints, we get maps $\CH_i(X)\otimes \Z\{i\} = \CH_i(X)\{i\}\to M^c(X)$: summing over $i$, we get a map in $DM$
\[\bigoplus_i \CH_i(X)\{i\} \to M^c(X),\]
which is an isomorphism for $X$ cellular by \cite[3.5]{kah99}. We can check that this map is natural for proper maps $f: X\to Y$. It suffices to check naturality of the maps (\ref{eq-bm}). By the definition of proper pushforward \cite[1.4]{fulton-intersection}, it amounts to check that the two horizontal compositions in the following diagram are the same. We check this by checking commutativity of the diagram:
\[\begin{tikzcd}
	{\Z\{i\}} & {M^c(Z)} & {M^c(X)} & {M^c(Y)} \\
	{\Z\{i\}} & {M^c(f(Z))} & {M^c(Y)}
	\arrow["{\cl_Z}", from=1-1, to=1-2]
	\arrow["\id"', from=1-1, to=2-1]
	\arrow["{i_*}", from=1-2, to=1-3]
	\arrow["{f_*}"', from=1-2, to=2-2]
	\arrow["{f_*}", from=1-3, to=1-4]
	\arrow["{f_*}"', from=1-3, to=2-3]
	\arrow["\id", from=1-4, to=2-3]
	\arrow["{\cl_{\deg(f)[f(Z)]}}"', from=2-1, to=2-2]
	\arrow["{i_*}"', from=2-2, to=2-3]
\end{tikzcd}\]
The left square commutes because the proper pushforward in motives of compact support involves the degree in the same way as the proper pushforward of cycles, as was observed in \cite[2.4]{nie}. The middle square commutes by functoriality of proper pushforwards. So the maps (\ref{eq-bm}) are natural.

Moreover, there is a degree-doubling cycle map from Chow groups to Borel--Moore homology groups. It is natural for proper maps, %
and it is an isomorphism on cellular varieties \cite[19.1.11.(b)]{fulton-intersection}, \cite[2.3]{rossello}. %
\item \begin{enumerate}[a)] 
\item By (\ref{item-mot1}) we obtain a commutative square in $DM$: %
\[\begin{tikzcd}
	{M^c(E)} & {M^c(Z)\oplus M^c(\widetilde X)} \\
	{\bigoplus_i \hbm_{2i}(E)\{i\}} & {\bigoplus_i \hbm_{2i}(Z)\{i\}\oplus \bigoplus_i \hbm_{2i}(\widetilde X)\{i\}}.
	\arrow[from=1-1, to=1-2]
	\arrow["\cong"', from=1-1, to=2-1]
	\arrow["\cong"', from=1-2, to=2-2]
	\arrow[from=2-1, to=2-2]
\end{tikzcd}\]
The cofiber of the top row is $M^c(X)$ by \cref{lem-mv}, so comparing cofibers gives the result. 

\item We now identify this cofiber under additional hypotheses. Suppose  $\hbm_{2i+1}(X)=0$. The long exact Borel--Moore homology sequence for the resolution of singularities gives a short exact sequence
\[0\to \hbm_{2i}(E)\to \hbm_{2i}(Z)\oplus \hbm_{2i}(\widetilde X)\to \hbm_{2i}(X) \to 0\]
which we can express as the triangle \[\hbm_{2i}(E)\to \hbm_{2i}(Z)\oplus \hbm_{2i}(\widetilde X)\to \hbm_{2i}(X) \to\] in $D(\Z)$.

When $\hbm_{2i}(X)$ is free, then the long exact sequence gives the four-term exact sequence
\[0\to \hbm_{2i+1}(X)\to \hbm_{2i}(E) \to \hbm_{2i}(Z) \oplus \hbm_{2i}(\widetilde X) \to \hbm_{2i}(X) \to 0.\]
By \cref{rmk-triangles}, we may equivalently express this as a triangle in $D(\Z)$:
\[\hbm_{2i}(E) \to \hbm_{2i}(Z)\oplus \hbm_{2i}(\widetilde X) \to \hbm_{2i}(X) \oplus \hbm_{2i+1}(X)[1] \to.\]
Passing to $DM$, summing and twisting these triangles appropriately, we get a triangle in $DM$:
\[\bigoplus_i \hbm_{2i}(E)\{i\} \to \bigoplus_i \hbm_{2i}(Z)\{i\}\oplus \bigoplus_i \hbm_{2i}(\widetilde X)\{i\} \to \bigoplus_i \hbm_{2i}(X)\{i\} \oplus \hbm_{2i+1}(X)\{i\}[1] \to.\]
This computes the cofiber in (\ref{item-cofiber}) and we are done.
\end{enumerate}
\item Since $X$ is complete, $M^c(X)=M(X)$ and $\hbm=H$. \qedhere
\end{enumerate}
\end{proof}
\end{thm}

Before we move on to applications of (\ref{item-mot2}), our main result, let us first say a word about the easier case (\ref{item-mot1}).

\begin{rmk} \label{rmk-cell}Since we can describe the Borel--Moore homology of a cellular variety in terms of its cellular decomposition (\cref{rmk-bmcell}), we conclude that if $X$ is cellular and $a_i$ is the number of cells of dimension $i$, then
\[M^c(X)\cong \bigoplus_i \Z^{a_i}\{i\}.\]
When $X$ is complete, this gives a description of the motive of $X$. Compare with %
\cite[Section 6]{karpenko}, see also \cite[Section 3]{brosnan}.
\end{rmk}

\begin{ex} As a simple application, %
we recover the motive with compact support of the blowup at $0$ of $\C^n$. Indeed, this is a %
cellular variety, with one cell in each dimension $i$, for $1\leq i \leq n$, so the motive with compact support is
\[M^c(\mathrm{Bl}_0(\C^n))\cong \bigoplus_{i=1}^n \Z\{i\}.\]
We could deduce this using more elementary properties: the $\mathbb{A}^1$-bundle map $ \mathrm{Bl}_0(\C^n)\to \bbP^{n-1}$ induces an isomorphism $M^c(\bbP^{n-1})\{1\}\to M^c(\Bl_0(\C^n))$, and $M^c(\bbP^{n-1})\{1\}\cong \bigoplus_{i=1}^n \Z\{i\}$.
\end{ex}

\begin{rmk} We could relax the assumptions in \cref{lem-mothom}.(\ref{item-dos}): indeed, we could ask for the diagram to be a homotopy pushout square of topological spaces with all maps proper and all varieties cellular, except possibly $X$. The same proof would apply, \emph{mutatis mutandis}.
\end{rmk}

\section{Applications}

We will now apply \cref{lem-mothom}.(\ref{item-mot2}) to different types of varieties: rational curves, toric surfaces, and some higher-dimensional toric varieties.

\subsection{Rational curves} 

By a \emph{curve} we mean a projective variety of dimension 1.

\begin{prop}\label{prop-curve}Let $C$ be a rational curve. Then there is an isomorphism in $DM$
\[M(C)\cong \Z \oplus \Z^{|E|-|Z|}[1] \oplus \Z\{1\},\]
where $Z$ is the singular locus of $C$ and $E$ is the exceptional locus of the resolution of singularities given by normalization.
\end{prop}

\begin{proof} Since $C$ is a curve, the normalization map $\widetilde C\to C$ is simultaneously a resolution of singularities. In particular, it is birational. %
Since $C$ is rational, there is a birational map from $\bbP^1$ to $C$. 
Therefore, $\widetilde C$ and $\bbP^1$ are birationally equivalent. But smooth projective birationally equivalent curves are isomorphic, %
so $\widetilde C\cong \bbP^1$. In conclusion, the resolution of singularities square looks like this:
\[\begin{tikzcd}
	E & {\bbP^1} \\
	Z & C
	\arrow["j", hook, from=1-1, to=1-2]
	\arrow["p"', from=1-1, to=2-1]
	\arrow[from=1-2, to=2-2]
	\arrow[hook, from=2-1, to=2-2]
\end{tikzcd}\]
where $E$, $Z$ are subvarieties of dimension zero, so both they and of course $\bbP^1$ are cellular.  %
The long exact homology sequence gives
\[0\to H_1(C)\to H_0(E)\stackrel{\alpha}{\to} H_0(Z)\oplus H_0(\bbP^1)\to H_0(C)\to 0.\]
We have $H_0(C)\cong \Z$. %
Identifying the homologies, we have $\alpha:\Z^{|E|}\to \Z^{|Z|+1}$, a morphism of free abelian groups with cokernel $\Z$, so its kernel is $\Z^{|E|-|Z|}\cong H_1(C)$. %

Looking higher up the long exact homology sequence, we see that $H_2(C)\cong \Z$. We are in the conditions of  \cref{lem-mothom}.(\ref{item-mot3}): applying it gives the result.
\end{proof}

\begin{rmk} \begin{enumerate}
\item The number of inverse images of $x\in C$ under the normalization map is the number of branches of $C$ at $x$. %
Thus, we may equivalently describe $|E|$ as the sum of the number of branches of $C$ at its singular points.
\item The motive $M(C)$ is pure Tate if and only if $|E|=|Z|$, i.e. all the singularities of $C$ are unibranch, or in other words, $H_1(C)=0$. Equivalently, $C$ is a fake projective line, i.e. it has the same homology as $\bbP^1$. We see that fake projective lines which are rational curves have the same motive as $\bbP^1$.
\item If $C$ is a projective toric curve (not necessarily normal), %
then it is rational and the theorem applies.%

\end{enumerate}
\end{rmk}

\begin{ex}
The projective toric curve $C=\Proj(\C[x,y,z]/(y^2z-x^3))$ has as unique singularity a cusp at the origin, %
so it is unibranch. Therefore, $M(C)\cong \Z\oplus \Z\{1\}$.
\end{ex}

\begin{ex}
The nodal projective curve $C=\Proj(\C[x,y,z]/(y^2z-x^3-x^2z))$ is also toric. It has a unique singularity with two branches. %
Thus, $H_1(C)\cong \Z$, and $M(C)\cong \Z \oplus \Z[1]\oplus \Z\{1\}$.%
\end{ex}

\begin{ex} For a non-toric example, we can take the %
ampersand curve,  %
i.e. the quartic with equation $(y^2-x^2)(x-1)(2x-3)=0$. 
It has three singularities, each of them with two branches, so $M(C)\cong \Z\oplus \Z^3[1]\oplus \Z\{1\}$.
\end{ex}

\begin{ex} For any $m\geq n\geq 1$ we can construct a rational curve $C$ with $|Z|=n$ and $|E|=m$. Since the structure of the singularities is not relevant for the motive, we can take them to be ordinary: we just define $C$ to be the pushout in schemes of $\Spec(\C)^{\sqcup n}\leftarrow \Spec(\C)^{\sqcup m} \to \bbP^1$, where the first map is any surjective map.
The pushout exists by \cite[3.9]{schwede}.
\end{ex}

\subsection{Cellular resolutions of singularities of toric varieties}
We now consider the question of the existence of a cellular resolution of singularities for normal toric varieties. Let us start by observing that normal projective toric surfaces always admit one.

\begin{prop} \label{prop-trees} Let $\Sigma$ be a fan in $N_\R\cong \R^2$. Let $\widetilde \Sigma$ be a smooth fan refining $\Sigma$ by adding rays %
such that the associated toric morphism $\xs\to X_\Sigma$ is a resolution of singularities with exceptional locus $E$. Then
the variety $E$ is a disjoint union of trees of projective lines, where by  ``tree'' we mean an intersection of finitely many copies of $\bbP^1$, where the first intersects the second transversally in a point, the second intercepts the third transversally in another point, etc., and those are the only intersections between them. %
In particular, $E$ is cellular.%
\end{prop}

\begin{proof}
First of all, such a refinement exists \cite[10.1.10]{cox}. The exceptional locus $E$ corresponds to the rays added to $\Sigma$. All the rays added to one of the cones in $\Sigma$ generate one such tree of projective lines: this is %
\cite[Page 47]{fulton-toric};
see also \cite[11.1.10, 11.2.2]{cox}. Considering all the rays added to all the cones gives then such a disjoint union of trees. \qedhere
\end{proof}

\begin{rmk} \label{rmk-homE}We have $H_0(E)=\Z^s$, where $s$ is the number of singular cones of dimension 2 in $\Sigma$. To compute $H_2(E)$, we can proceed as follows. The variety $E$ can be presented as an iterated (homotopy) pushout: The (homotopy) pushout of $\bbP^1 \leftarrow \Spec(\C) \to \bbP^1$ gives a tree of two projective lines, to which we can add another line via another (homotopy) pushout, etc. A Mayer--Vietoris argument then proves that  $H_2(E)=\Z^k$, where $k$ is the number of rays added to $\Sigma$. Therefore, \cref{lem-mothom} gives that $M(E)=\Z^s\oplus \Z\{1\}^k$. %
\end{rmk}

\begin{prop} \label{rmk-surface-cellular} Any normal projective toric surface $X$ admits a cellular resolution of singularities. 
\begin{proof}
The singular locus consists of finitely many points, %
and as observed in \cref{prop-trees}, we can find a resolution $\widetilde X\to X$ with exceptional locus cellular; $\widetilde X$ is smooth projective toric, hence also cellular.
\end{proof}
\end{prop}

If we take a fan $\Sigma$ and refine it to a fan $\widetilde\Sigma$ so as to obtain the toric resolution of singularities $\xs\to X_\Sigma$, then if $\xs$ is not projective, it may not be cellular. Let us now state a lemma that guarantees that it will be cellular under more general hypotheses than completeness of the fan:

\begin{lemma} \label{lemma-oneparam} %
Let $\widetilde\Sigma$ be a smooth fan in $N_\R$. Suppose $\xs$ is quasiprojective. %
If there exists a $u\in N$ such that:
\begin{enumerate}[a)]
\item \label{cond-supp} $|\widetilde\Sigma|$ contains $|\widetilde\Sigma|+u$, where $|\widetilde\Sigma|$ denotes the support of $\widetilde\Sigma$, 
\item \label{cond-mingen} $\langle m_i,u \rangle\not=0$ for all $i$, where $\{m_i\}_i\subset M$ is a set consisting of semigroup generators of all the $\sigma^\vee\cap M$, for all the maximal cones $\sigma\in \widetilde\Sigma$,
\end{enumerate}
then $\xs$ is cellular.
\end{lemma}

\begin{proof}Denote by $T_N$ the torus of $\xs$. For any $u\in N$, consider the one-parameter subgroup $\lambda^u:\C^*\to T_N$ it defines, which gives a $\C^*$-action on $\xs$. %

Condition \ref{cond-supp}) is equivalent to the limits $\lim_{t\to 0} t\cdot x$ existing for all $x\in \xs$, where $t\cdot x$ denotes the $\C^*$-action above \cite[5.2.1]{nicaise-payne}. %

Condition \ref{cond-mingen}) is equivalent to the one-parameter subgroup being \emph{regular}, which means that the fixed point set of the action of $T_N$ on $\xs$ coincides with that of its one-parameter subgroup \cite[4.2.5]{carrell}. %
Since $T_N$ fixes only finitely many points, this guarantees that its one-parameter subgroup does so, too, and thus we are in the conditions of \cref{ex-cellular}.(\ref{item-cellular-3}).
\end{proof}

\begin{rmk} \label{rmk-cond}
\begin{enumerate}
\item \label{item-quasiproj}The toric variety of a full-dimensional lattice polyhedron is quasiprojective \cite[7.1.10]{cox}. %
If $\widetilde\Sigma$ is the refinement of some fan $\Sigma$ and $X_\Sigma$ %
is quasiprojective (resp. projective), then $\xs$ is also quasiprojective (resp. projective), by \cite[7.16]{hartshorne} and \cite[13.102]{gortz}. %
\item Condition \ref{cond-supp}) is satisfied for any $u\in |\Sigma|\cap N$ if $|\Sigma|$ is convex: for example, if $\Sigma$ is generated by a single one-dimensional cone, so that the toric variety is affine. %
\item The fan $\Sigma$ generated by $\cone(e_1,-e_1+e_2)$ and $\cone(e_1,-e_1-e_2)$ in $\R^2$ is not convex but it satisfies \ref{cond-supp}): the vectors that satisfy it are those in $\cone(e_1+e_2,e_1-e_2)$ which is the image of the ``missing cone'' $\R^2\setminus \int(|\Sigma|)$ under the central symmetry at the origin, where $\int$ denotes the topological interior.
\item An example of a fan that does not satisfy Condition \ref{cond-supp}) for any $u\in N$ is the fan generated by the cones $\cone(e_1,e_2)$ and $\cone(-e_1,-e_2)$ in $\R^2$.  %
Therefore, this toric variety does not admit any action of a one-parameter subgroup for which the limits $\lim_{t\to 0}t\cdot x$ exist for all $x\in X$.
\item We now give an example of a fan that satisfies Condition \ref{cond-supp}) for some $u$, but does not satisfy Condition \ref{cond-mingen}) for any vector.  %
Take the fan $\Sigma$ in $\R^3$ generated by the two cones $\sigma_1=\cone(-e_1+e_2+e_3, -e_1-e_2+e_3,e_3)$ and $\sigma_2=\cone(e_1+e_2+e_3,e_1-e_2+e_3,e_3)$. Then the only $u\in N$ that satisfy Condition $a)$ are those in $L\coloneqq \cone(e_3)\cap N$. Indeed, a reformulation of  Condition $a)$ is that $|\Sigma|$ is star-shaped around $u$ (\cite[5.2.1]{nicaise-payne}), but this only happens for the vectors in $L$. We find that $\sigma_1^\vee=\cone(e_1^*+e_3^*,-e_1^*+e_2^*, -e_1^*-e_2^*)$. 
The second covector, for example, is perpendicular to $\cone(e_3)$, so no vector in $L$ can satisfy Condition $b)$.
\item \label{cond-auto}If there exists a cone of maximal dimension in $N_\R$ consisting of vectors satisfying Condition $a)$, then there exists a vector $u\in N$ satisfying both $a)$ and $b)$. Indeed, the vectors that do not satisfy $b)$ are a finite union of hyperplanes: they cannot cover an entire maximal cone. In particular, if  the support $|\Sigma|$ is convex, then there exists a vector $u\in N$ satisfying both conditions. In particular, smooth affine toric varieties are cellular.%
\item By the same reasoning, if the region $\nu\coloneqq N_\R\setminus  \int(|\Sigma|)$ consists of a cone of maximal dimension, then there exists a vector $u\in N$ satisfying both conditions. In this case, it is not hard to see that Condition a) is satisfied for the maximal cone of vectors which are in $c(\nu)$, where $c$ is the central symmetry with respect to the origin.
\end{enumerate}
\end{rmk}

\begin{prop} \label{prop-affinesurface}Any normal affine toric surface $X$ admits a cellular resolution of singularities.
\begin{proof} Suppose $X$ is the toric variety of the cone $\sigma$ in $N_\R\cong \R^2$. The singular locus of $X$ consists of finitely many points, %
and as recalled in \cref{prop-trees}, we can refine $\sigma$ to a fan $\widetilde\Sigma$ so that the resolution of singularities $\xs\to X$ has cellular exceptional locus. We now apply the previous lemma. We use \cref{rmk-cond}: by (\ref{item-quasiproj}), $\xs$ is quasiprojective, so since $|\sigma|=|\widetilde\Sigma|$ is convex, we have that $\xs$ is cellular by (\ref{cond-auto}).
\end{proof}
\end{prop}

\subsection{Motives of toric surfaces} We can now describe the motives of projective normal toric surfaces, and the motive with compact support of many other normal toric surfaces, notably all the affine ones. We begin by recalling the description of the Borel--Moore homology of normal toric surfaces.

\begin{defn} The \emph{index} of a fan in $N_\R\cong \R^2$ with $r$ rays is the greatest common divisor of the determinants $\det(u_i,u_j)$, $1\leq i<j\leq r$, where the $u_i$ are the minimal ray generators of the fan.
\end{defn}

\begin{rmk} Alternatively, the index of a fan is the index of the lattice spanned by the $u_i$ inside $N\cong \Z^2$. %
\end{rmk}

\begin{prop} \label{prop-homology-jordan}Let $\Sigma$ be a fan in $N_\R \cong \R^2$. The Borel--Moore homology of $X_\Sigma$ is:
\[
\hbm_0(X_\Sigma)=\begin{cases}\Z & \text{if }\Sigma \text{ is complete,} \\ 0 & \text{else,}\end{cases} \hspace{.5cm}
\hbm_1(X_\Sigma)=\begin{cases}0 & \text{if }\Sigma \text{ is complete,} \\ \Z^{-d_2+d_1-d_0} & \text{else,}\end{cases} \]
\[
\hbm_2(X_\Sigma)=\Z^{d_1-s}\oplus \Z/m, \hspace{.5cm}
\hbm_3(X_\Sigma)=\Z^{2-s}, \hspace{.5cm}
\hbm_4(X_\Sigma)=\Z, \hspace{.5cm}
\]
where $d_i$ is the number of cones of dimension $i$, %
$s$ is the dimension of the subspace of $N_\R$ spanned by the rays and $m$ is the index of $\Sigma$.
\begin{proof} This follows from the toric spectral sequence of \cite{jordan}, as explained in Theorem 3.4.2 therein.
\end{proof}
\end{prop}

\begin{rmk} If the fan $\Sigma$ in $N_\R\cong \R^2$ is non-degenerate, i.e. $s=2$, then we have that $\hbm_0(X_\Sigma)$ is free, $\hbm_3(X_\Sigma)=0$ and $\hbm_4(X_\Sigma)$ is free, so the conditions of \cref{lem-mothom}.(\ref{item-cofiber}) on the homology are satisfied.
\end{rmk}

\begin{cor}\label{rmk-homology-toric} Let $\Sigma$ be a complete fan in $N_\R \cong \R^2$. The homology of $X_\Sigma$ is concentrated in degrees $0,2,4$, where it takes the following values: $H_0(X_\Sigma)=H_4(X_\Sigma)=\Z$, and $H_2(X_\Sigma)=\Z^{r-2}\oplus \Z/m$,
for $r$ the number of rays in $\Sigma$ and $m$ the index of $\Sigma$.
\end{cor}

We can now describe the motive of any projective normal toric surface:

\begin{prop} \label{thm-motive-surface} Let $\Sigma$ be a complete fan in $N_\R\cong \R^2$. Then
\[M(X_\Sigma)\cong \Z \oplus \Z^{r-2}\{1\} \oplus \Z/m\{1\} \oplus \Z\{2\},\]
where $r=|\Sigma(1)|$ and $m$ is the index in $\Z^2$ of the lattice spanned by the minimal generators of the rays of $\Sigma$.
\begin{proof}
As observed in \cref{rmk-surface-cellular}, $X_\Sigma$ admits a cellular resolution of singularities. The homology groups of $X_\Sigma$ are given in \cref{rmk-homology-toric}: since the odd ones are zero, we are in the conditions of \cref{lem-mothom}.(\ref{item-mot2}). The result follows. %
\end{proof}
\end{prop}

\begin{rmk}
\begin{enumerate}
\item If $\Sigma$ is smooth, then $m=1$, so if $\Sigma$ is moreover complete, then:
\[M(X_\Sigma)\cong \Z \oplus \Z\{1\}^{r-2} \oplus \Z\{2\}.\]
\item 
The motive of $X_\Sigma$ as in \cref{thm-motive-surface} is pure Tate if and only if $m=1$. %
In other words: the motive of $X_\Sigma$ is not pure Tate if and only if all the 2-dimensional cones of $\Sigma$ are singular (in particular, $X_\Sigma$ must have at least three singular points), and the greatest common divisor of their multiplicities is $\geq 2$. %
\end{enumerate}
\end{rmk}

Let us now give some concrete examples.

\begin{ex} Let $k\geq 2$ and consider the singular fan $\Sigma$ in $\R^2$ spanned by the generators $-e_1,e_2,ke_1-e_2$. %
The resulting projective toric surface $X_\Sigma$ is a weighted projective space $\bbP(1,1,k)$ \cite[3.1.17]{cox}. 
We refine $\Sigma$ to a smooth fan $\widetilde\Sigma$ by adding the vector $(1,0)$, which induces the resolution of singularities $X_{\widetilde\Sigma}\to X_\Sigma$; the variety $\xs$ is the Hirzebruch surface $\mathcal{H}_k$ \cite[3.1.16]{cox}. %
We picture the $k=2$ case below.
\[\begin{tikzcd}
	& {} &&&&& {} \\
	{} & {\bullet} & {} & { } & { } & {} & \bullet \\
	&&& {} &&&&& {}
	\arrow[from=2-2, to=1-2]
	\arrow[from=2-2, to=2-1]
	\arrow[from=2-2, to=2-3]
	\arrow[from=2-2, to=3-4]
	\arrow["{\text{refines}}", from=2-4, to=2-5]
	\arrow["{(0,1)}"', from=2-7, to=1-7]
	\arrow["{(-1,0)}"', from=2-7, to=2-6]
	\arrow["{(2,-1)}", from=2-7, to=3-9]
\end{tikzcd}\]
We get
\[M(\bbP(1,1,k))\cong \Z\oplus \Z\{1\} \oplus \Z\{2\}.\]
In particular, $M(\bbP(1,1,k))$ is pure Tate. Notice also that $\bbP(1,1,k)$, $k\geq 2$ is a so-called fake projective plane: a complex algebraic variety with the same homology as $\bbP^2$ %
which is nonetheless not isomorphic to it (it is not even smooth). We observe that it has not only the same homology as $\bbP^2$, but also the same motive. In fact, it immediately follows from \cref{thm-motive-surface} that any fake projective plane which is a normal toric surface has the same motive as $\bbP^2$. We generalize this example in \cref{ex-weighted-general}.

Notice also that all the $\mathcal{H}_k$ have the same homology, %
and thus also the same motive, even though they are not pairwise isomorphic.
\end{ex}

Let us now give a more complicated example, where there is torsion in the second homology group.

\begin{ex} %
Consider the complete, singular fan $\Sigma$ in $\R^2$ spanned by the generators $u_1=e_2,u_2=-2e_1-e_2, u_3=2e_1-e_2$. All three 2-dimensional cones in it are singular, so there are three singular points. %
We resolve them by adding the rays spanned by the vectors $e_1,-e_1,e_1-e_2,-e_2,-e_1-e_2$.
\[\begin{tikzcd}
	&& {} &&&&&& {} \\
	& {} & \bullet & {} & { } && { } && \bullet \\
	{} & {} & {} & {} & {} && {} &&&& {}
	\arrow[from=2-3, to=1-3]
	\arrow[from=2-3, to=2-2]
	\arrow[from=2-3, to=2-4]
	\arrow[from=2-3, to=3-1]
	\arrow[from=2-3, to=3-2]
	\arrow[from=2-3, to=3-3]
	\arrow[from=2-3, to=3-4]
	\arrow[from=2-3, to=3-5]
	\arrow["{\text{refines}}"', from=2-5, to=2-7]
	\arrow[from=2-9, to=1-9]
	\arrow[from=2-9, to=3-7]
	\arrow[from=2-9, to=3-11]
\end{tikzcd}\]
The resolution of singularities diagram now looks like:
\[\xymatrix{\bbP^1\sqcup \bbP^1\sqcup F \ar[r] \ar[d] & \xs \ar[d] \\ \Spec(\C) \sqcup \Spec(\C) \sqcup \Spec(\C) \ar[r] & X_\Sigma},\]
where $F$ is %
a tree of three projective lines. 
The index of %
$\Sigma$ is 2, therefore
\[M(X_\Sigma)\cong \Z\oplus \Z\{1\}\oplus \Z/2\{1\} \oplus \Z\{2\}.\]
In particular, $M(X_\Sigma)$ is not pure Tate.
\end{ex}

Let us now describe the motive with compact support of normal affine toric surfaces. From \cref{prop-homology-jordan}, we deduce:

\begin{cor}\label{rmk-homology-affine} Let $\sigma$ be a two-dimensional cone in $N_\R \cong \R^2$. The Borel--Moore homology of $U_\sigma$ takes the values:
\[\hbm_i(U_\sigma) \cong \begin{cases}
0 & \text{if } i=0,1,3, \\ \Z/m & \text{if } i=2, \\ \Z & \text{if } i=4,
\end{cases}\]
where $m$ is the determinant of the two minimal ray generators of $\sigma$.
\end{cor}

\begin{prop}  \label{prop-mc-affine}Let $\sigma$ be a two-dimensional cone in $N_\R\cong \R^2$. The motive with compact support of $U_\sigma$ is given by
\[M^c(U_\sigma)\cong \Z/m\{1\} \oplus \Z\{2\}\]
where $m$ is the determinant of the two minimal ray generators of $\sigma$. %
\begin{proof} By \cref{prop-affinesurface}, $U_\sigma$ admits a cellular resolution of singularities. The homology groups of $U_\sigma$ are given in \cref{rmk-homology-affine}: since the odd ones are zero, we are in the conditions of \cref{lem-mothom}.(\ref{item-mot2}). The result follows.
\end{proof}
\end{prop}

\begin{ex} \label{ex-quasiproj} Let $d\geq 1$ and take $\sigma=\cone(d e_1-e_2,e_2)$. %
Then $U_\sigma$ is the \emph{rational normal cone} of degree $d$. Its motive with compact support is
\[M^c(U_\sigma)\cong \Z/d\{1\} \oplus \Z\{2\}.\]
\end{ex}

Let us give an example of a motive with compact support of a toric surface which is neither projective nor affine:

\begin{ex} Take the polyhedron in $M_\R\cong \R^2$ which is the convex region in the upper half-plane delimited by the segment $[(0,0),(1,0)]$ and two half-lines: one starting from $(0,0)$ and having tangent vector $(-1,k)$, and the other starting from $(1,0)$ and having tangent vector $(1,d)$, for some $k,d\geq 2$. %
Its normal fan $\Sigma$ is generated by the cones $\sigma_1=\cone(ke_1+e_2,e_2)$, $\sigma_2=\cone(-de_1+e_2,e_2)$ in $N_\R\cong \R^2$. %
The variety $X_\Sigma$ has exactly two singular points, of multiplicities $k$ and $d$. Resolving singularities by adding rays gives a smooth fan $\widetilde\Sigma$ which also comes from a polyhedron, e.g. for $k=d=2$ we need to add two rays, and the polyhedron whose normal fan is $\widetilde\Sigma$  is the infinite convex region bounded by the polyhedron pictured below.
\[\begin{tikzcd}
	{} &&&&& {} \\
	\\
	& \bullet &&& \bullet \\
	&& 0 & \bullet
	\arrow[no head, from=1-1, to=3-2]
	\arrow[no head, from=3-2, to=4-3]
	\arrow[no head, from=3-5, to=1-6]
	\arrow[no head, from=4-3, to=4-4]
	\arrow[no head, from=4-4, to=3-5]
\end{tikzcd}\]
Therefore, $\xs$ is quasiprojective, and the support of $\widetilde\Sigma$ is convex, so $\xs$ is cellular. We can thus describe the motive with compact support of $X_\Sigma$:
\[M^c(X_\Sigma)\cong \Z\{1\} \oplus (\Z/\gcd(k,d))\{1\} \oplus \Z\{2\}.\]
\end{ex}

\subsection{Motives of toric varieties of higher dimension}

As for surfaces, the Borel--Moore homology of a normal toric threefold can also be completely read off the $E^2$-page of the toric spectral sequence of \cite{jordan} for threefolds, see 3.5.1 therein for an explicit expression. This $E^2$-page can be computed using the \emph{torhom} package for Maple \cite{torhom}.

\begin{ex}Let $\sigma$ be the cone spanned by the vectors $e_1,e_2,e_1+e_3, e_2+e_3$ in $\R^3$. Let $U_\sigma$ be the corresponding affine toric variety. The facets %
of $\sigma$ are all smooth, but $\sigma$ itself is not, so the singular locus consists of a single point. %
We can describe $U_\sigma$ as the quadric $V(xy-zw)\subset \C^4$ \cite[1.2.9]{cox}. We can resolve the singularity by dividing $\sigma$ in two, adding the facet spanned by $e_1, e_2+e_3$: denote the corresponding fan by $\widetilde\Sigma$. The exceptional locus is  $\bbP^1$ \cite[11.1.12]{cox}. Since $|\widetilde\Sigma|$ is convex, by \cref{rmk-cond}.(\ref{cond-auto}), we can find a vector $u\in N$ satisfying the conditions of \cref{lemma-oneparam}; %
more explicitly, we could take $u=(1,2,1)$. Since moreover $\xs$ is quasiprojective by \cref{rmk-cond}.(\ref{item-quasiproj}), we have that $\xs$ is cellular. Therefore, $X_\Sigma$ admits a cellular resolution of singularities.

Let us compute the Borel--Moore homology of $U_\sigma$ %
using \emph{torhom}. 
We obtain:
\[\hbm_i(U_\sigma)=\begin{cases}\Z & \text{if } i=3,4,6, \\0 & \text{else.}\end{cases}\]
We are thus in the conditions of \cref{lem-mothom}.(\ref{item-mot3}). We apply it to get
\[M^c(U_\sigma)\cong \Z\{1\}[1] \oplus \Z\{2\} \oplus \Z\{3\}.\]
\end{ex}

\begin{ex} Let $\sigma$ be the cone spanned by the vectors $e_1+e_2$, $-e_1+e_2$, $e_3$ in $N_\R\cong\R^3$ and let $\Sigma$ be the fan it generates.  This is the product of the fan generated by $\cone(e_1+e_2, -e_1+e_2)\subset \R^2$ and the one generated by $\cone(e_3)\subset \R$, so the affine toric variety $U_\sigma$ can be described as
\[\Spec(\C[x,y,z]/(xy-2z))\times \C.\] %
The minimal singular cone of $\Sigma$ is generated by $e_1+e_2,-e_1+e_2$, so the singular locus is $Z\cong \C$.

We can refine $\Sigma$ by adding $\cone(e_2,e_3)$, which generates a fan $\widetilde\Sigma$. The exceptional locus is then $E\cong \bbP^1\times \C$. We now compute the Borel--Moore homology of $U_\sigma$ using \emph{torhom}, getting:
\[\hbm_i(U_\sigma)=\begin{cases}
\Z/2 &\text{if } i=4,\\
\Z &\text{if } i=6,\\
0 & \text{else.}\\
\end{cases}\]
Since $|\widetilde\Sigma|$ is convex, by \cref{rmk-cond}.(\ref{cond-auto}), we can find a vector $u\in N$ satisfying the conditions of \cref{lemma-oneparam}; more explicitly, we could take $u=(1,2,1)$. Also, $\xs$ is quasiprojective by \cref{rmk-cond}.(\ref{item-quasiproj}). Therefore, $\xs$ is cellular, and thus, $X_\Sigma$ admits a cellular resolution of singularities. We can therefore apply \cref{lem-mothom}, getting
\[M^c(U_\sigma)\cong \Z/2\{2\} \oplus \Z\{3\}.\]
\end{ex}

\begin{ex} Consider the cube with vertices $(\pm 1, \pm 1, \pm 1)$ in $\R^3$. Let $N$ be the lattice generated by the vertices. Let $\Sigma$ be the complete fan in $N_\R$ whose six maximal cones are the cones over the faces of the cube. %
The region bounded inside the cube is the polar polytope of another lattice polytope containing zero \cite[2.3.11]{cox}, so $X_\Sigma$ is projective. %
Its singular locus consists of the six points corresponding to the six maximal cones. 

We can resolve the singularities of $X_\Sigma$ by refining the maximal cones, dividing each of them by a diagonal.  %
The exceptional locus is isomorphic to $(\bbP^1)^{\sqcup 6}$. Moreover, $\xs$ is a smooth and projective (by \cref{rmk-cond}.(\ref{item-quasiproj})) toric variety, so it is cellular. Thus, $X_\Sigma$ admits a cellular resolution of singularities.

We now find the homology of $X_\Sigma$. From \cite[Pages 105-106]{fulton-toric}, we read the homology groups in all degrees but 1 and 5. In those degrees, we use \cite[3.5.1]{jordan} to conclude that
\[H_i(X_\Sigma)=\Z \text{ for } i=0,2,6, \hspace{.5cm} H_3(X_\Sigma)=\Z^2, \hspace{.5cm} H_4(X_\Sigma)=\Z^5, \hspace{.5cm} H_1(X_\Sigma)=H_5(X_\Sigma)=0.\]
Since all the groups in even grading are free, we can apply \cref{lem-mothom}.(\ref{item-mot2}):
\[M(X_\Sigma) \cong \Z\oplus \Z\{1\} \oplus \Z^2\{1\}[1] \oplus \Z^5\{2\} \oplus \Z\{3\}.\]
\end{ex}

 \begin{ex} The exceptional locus of a toric resolution of singularities need not be cellular. For example, in $\R^3$ take the fan generated by the three cones $\sigma_1=\cone(e_1,e_2,e_1+e_2\pm e_3)$, $\sigma_2=\cone(e_2,-e_1-e_2, -e_2\pm e_3)$, $\sigma_3=\cone(-e_1-e_2, e_1, -e_2\pm e_3)$. %
The singular locus consists of the three points corresponding to the three maximal cones. We can resolve the singularities by subdivision: add the faces generated by the first two cone generators in each of the cones. The exceptional locus consists of three projective lines, where the first/second/third intersects the second/third/first each in a different point%
: this is not a cellular variety. %
\end{ex}

 \begin{ex} \label{ex-weighted-general}Let $\Sigma$ be the complete fan in $N_\R\cong \R^n$ generated by $e_1,\dots,e_n, -e_1-\cdots -ke_n$ for some $k\geq 2$, so $X_\Sigma$ is the weighted projective space $\bbP(1,\dots,1,k)$. %
The singular locus is a single point, and we can resolve this singularity by adding the ray $-e_n$, generating a fan $\widetilde\Sigma$; the variety $\xs$ is projective by \cref{rmk-cond}.(\ref{item-quasiproj}). The exceptional divisor is $\bbP^{n-1}$. The homology of weighted projective spaces is the same as that of projective spaces \cite{kawasaki}, so applying \cref{lem-mothom} proves that the motive of $\bbP(1,\dots,1,k)$ is the same as the motive of $\bbP^n$. We raise the question: does every fake projective space have the same motive as the projective space of the same dimension?
\end{ex}

\begin{ex}\label{ex-bbfk} We consider the example of \cite[3.5.4]{jordan}, taken from \cite[3.5]{diviseurs}. Take the complete fan $\Sigma$ in $N_\R\cong \R^3$ whose maximal cones $\sigma_i$, $i=1,\dots,5$ are defined as follows: let $v_1=e_1+e_2+e_3$, $v_2=-e_1+e_2+e_3$, $v_3=-e_2+e_3$, and $v_i=v_{i-3}-2e_3$ for $i=4,5,6$. We let $\sigma_i$ be the cones spanned by the the following lists of vectors, given in order:
\[(v_1,v_2,v_3), \hspace{.5cm} (v_1,v_2,v_5,v_4), \hspace{.5cm} (v_2,v_3,v_6,v_5), \hspace{.5cm}(v_3,v_1,v_4,v_6), \hspace{.5cm}(v_4,v_5,v_6).\] %
Note that $\Sigma$ is the normal fan of the polytope $\Conv(0,(0,-1,-1),(\pm 2,1,-1),(0,0,-2))$, 
so $X_\Sigma$ is projective. We compute its homology using \emph{torhom}, getting $H_2(X_\Sigma)=\Z\oplus \Z/2$ and $H_3(X_\Sigma)=\Z$: this shows that the conditions in the homology groups in  \cref{lem-mothom}.(\ref{item-mot3}) are not always satisfied.

We can also determine the singular locus of $X_\Sigma$. The minimal singular cones of $\Sigma$ are the facets generated by the following pairs of vectors, displayed in order: %
\[(v_1,v_2),\hspace{.5cm} (v_1,v_4),\hspace{.5cm} (v_2,v_5),\hspace{.5cm} (v_3,v_6),\hspace{.5cm} (v_4,v_5).\]
Therefore, the singular locus $Z$ of $X_\Sigma$ is the intersection of five projective lines $P_i$ with only three points $x_i$ (corresponding to the cone $\sigma_i$) of common intersection, as displayed in \cref{fig:sing-notcell}.
\begin{figure}
\centering
\def\svgscale{0.3}
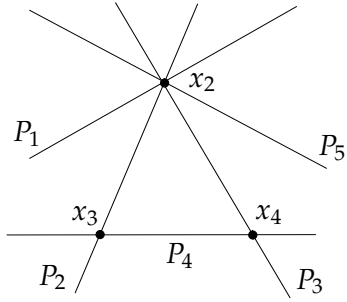
\caption{The singular locus of the toric variety in \cref{ex-bbfk}}
\label{fig:sing-notcell}
\end{figure}
We see that $Z$ is not cellular, as its first Borel--Moore homology group is not zero: $Z$ is complete, and an argument as in \cref{rmk-homE} proves that $H_1(Z)=\Z$. Therefore, $X_\Sigma$ does not admit a cellular resolution of singularities.
\end{ex}

\bibliographystyle{alpha}
\bibliography{motivic.bib}
\end{document}

%% file: graphic.pdf_tex
\begingroup%
  \makeatletter%
  \providecommand\color[2][]{%
    \errmessage{(Inkscape) Color is used for the text in Inkscape, but the package 'color.sty' is not loaded}%
    \renewcommand\color[2][]{}%
  }%
  \providecommand\transparent[1]{%
    \errmessage{(Inkscape) Transparency is used (non-zero) for the text in Inkscape, but the package 'transparent.sty' is not loaded}%
    \renewcommand\transparent[1]{}%
  }%
  \providecommand\rotatebox[2]{#2}%
  \newcommand*\fsize{\dimexpr\f@size pt\relax}%
  \newcommand*\lineheight[1]{\fontsize{\fsize}{#1\fsize}\selectfont}%
  \ifx\svgwidth\undefined%
    \setlength{\unitlength}{525bp}%
    \ifx\svgscale\undefined%
      \relax%
    \else%
      \setlength{\unitlength}{\unitlength * \real{\svgscale}}%
    \fi%
  \else%
    \setlength{\unitlength}{\svgwidth}%
  \fi%
  \global\let\svgwidth\undefined%
  \global\let\svgscale\undefined%
  \makeatother%
  \begin{picture}(1,1)%
    \lineheight{1}%
    \setlength\tabcolsep{0pt}%
    \put(0,0){\includegraphics[width=\unitlength,page=1]{graphic.pdf}}%
    \put(0.54333845,0.6673427){\color[rgb]{0,0,0}\makebox(0,0)[lt]{\lineheight{1.25}\smash{\begin{tabular}[t]{l}$x_2$\end{tabular}}}}%
    \put(0.26621477,0.34805254){\color[rgb]{0,0,0}\makebox(0,0)[lt]{\lineheight{1.25}\smash{\begin{tabular}[t]{l}$x_3$\end{tabular}}}}%
    \put(0.70118086,0.35095154){\color[rgb]{0,0,0}\makebox(0,0)[lt]{\lineheight{1.25}\smash{\begin{tabular}[t]{l}$x_4$\end{tabular}}}}%
    \put(0,0){\includegraphics[width=\unitlength,page=2]{graphic.pdf}}%
    \put(0.12789285,0.53714983){\color[rgb]{0,0,0}\makebox(0,0)[lt]{\lineheight{1.25}\smash{\begin{tabular}[t]{l}$P_1$\end{tabular}}}}%
    \put(0.85095119,0.50346414){\color[rgb]{0,0,0}\makebox(0,0)[lt]{\lineheight{1.25}\smash{\begin{tabular}[t]{l}$P_5$\end{tabular}}}}%
    \put(0.1888945,0.19512234){\color[rgb]{0,0,0}\makebox(0,0)[lt]{\lineheight{1.25}\smash{\begin{tabular}[t]{l}$P_2$\end{tabular}}}}%
    \put(0.79793228,0.18459342){\color[rgb]{0,0,0}\makebox(0,0)[lt]{\lineheight{1.25}\smash{\begin{tabular}[t]{l}$P_3$\end{tabular}}}}%
    \put(0.48927906,0.24569136){\color[rgb]{0,0,0}\makebox(0,0)[lt]{\lineheight{1.25}\smash{\begin{tabular}[t]{l}$P_4$\end{tabular}}}}%
  \end{picture}%
\endgroup%